\UseRawInputEncoding
\documentclass{article}
\usepackage{amsfonts}
\usepackage{amsfonts}
\usepackage{amsfonts}
\usepackage{amsfonts}
\usepackage{amsfonts}
\usepackage{amsfonts}
\usepackage{amsfonts}
\usepackage{mathrsfs}
\usepackage{CJK,amssymb,amsmath,amsfonts,amsthm,graphicx,enumerate,amscd}
\usepackage{bbm}
\usepackage{amssymb}
\usepackage{graphicx,verbatim}
\usepackage{amscd}
\usepackage{amsmath}
\usepackage{amsfonts}
\usepackage{graphicx}
\usepackage{latexsym}
\usepackage{multirow}
\usepackage{bigdelim}
\usepackage{color}

\newtheorem{thm}{Theorem}[section]
\newtheorem{lem}{Lemma}[section]
\newtheorem{cor}{Corollary}[section]
\newtheorem{prop}{Proposition}[section]

\newtheorem{rmk}{Remark}[section]

\begin{document}
\title{Grassmannian perspectives  of classical Lie groups and Cartan involutions}
\author{Yunxia Chen \& Naichung Conan Leung}
\date{}
\maketitle

\begin{abstract}
Classical noncompact reductive Lie group $G$ admits a compactification $\overline{G}$ as a Riemannian symmetric space by He. First, we provide a unified construction of these compactifications via Grassmannian geometry and realize the group structures in terms of the geometry of configurations of linear subspaces.

Second, we show that the Cartan involution $\rho$ on $G$  extends uniquely to an isometric involution $\bar{\rho}$ on $\overline{G}$ and $\overline{G}^{\bar{\rho}} = G^{\rho} = K$, the maximal compact subgroup of $G$.

Third, we show that $\eta(g) = \rho(g)^{-1}$  extends uniquely to an isometric involution $\bar{\eta}$ on $\overline{G}$ and $\overline{G}^{\bar{\eta}} = G_c/K$, the compact symmetric space dual to $(G^\eta)_0 = G/K$. This provides a natural generalization of the classical Borel embeddings $G/K \hookrightarrow G_c/K$. Furthermore, $K$ and $G_c/K$ form a complementary  pair of reflective submanifolds in $\overline{G}$.
\end{abstract}
\section{Introduction}
Throughout this paper, let $G$ be a real reductive Lie group of classical type\footnote{The classical Lie groups here are in Howe's sense, i.e. we refer to the matrix groups  without requiring their
determinants equal to one. Specifically, we work with the groups $GL(n, \mathbb{R})$, $GL(n, \mathbb{C})$, $GL(n, \mathbb{H})$, $O(p,q)$, $U(p,q)$, $Sp(p,q)$,
$Sp(n,\mathbb{R})$, $Sp(n,\mathbb{C})$, $O(n,\mathbb{C})$, $O^{*}(2n)$.}, and let $V \cong \mathbb{F}^n$ be a vector space\footnote{$\mathbb{F}$-module when $\mathbb{F}=\mathbb{H}$.} over $\mathbb{F}$, where $\mathbb{F}$ is either $\mathbb{R}$, $\mathbb{C}$, or $\mathbb{H}$.
By classification \cite{Har}\cite{K}, such a group $G$  can be realized as the automorphism group $G=Aut(V,b)$ of a (possibly zero) $\mathbb{R}$-bilinear form $b$  on $V$. Explicitly, we have the following correspondence\footnote{Refer to section 1.1 in \cite{Har} for the standard $\mathbb{R}$-bilinear forms and the groups defined by these forms.}:
\vspace{-8pt}
\begin{table}[htbp]
\centering
\caption{The groups defined by $\mathbb{R}$-bilinear forms}
\begin{tabular}{|c|c|c|c|c|c|}
  \hline
  $\mathbb{F}$ $\backslash$ $b$    & $b=0$ & \small{symmetric}  & \small{skew}  & \small{Hermitian  symmetric} & \small{Hermitian skew}  \\  \hline
  $\mathbb{R}$  & $GL(n, \mathbb{R})$ & $O(p,q)$ & $Sp(n,\mathbb{R})$ &  &  \\  \hline
  $\mathbb{C}$ & $GL(n, \mathbb{C})$ &  $O(n,\mathbb{C})$ & $Sp(n,\mathbb{C})$ & $U(p,q)$ & $U(p,q)$ \\ \hline
  $\mathbb{H}$ & $GL(n, \mathbb{H})$  &  &  & $Sp(p,q)$ & $O^{*}(2n)$ \\ \hline
\end{tabular}
\end{table}

In \cite{HL},  it is established that every classical compact Lie group admits a Grassmannian characterization via the following explicit isomorphisms:
\[
O(n) \cong \{L \cong \mathbb{R}^n \subset \mathbb{R}^{n}\oplus \mathbb{R}^{n}: h\oplus-h|_{L}=0\}
\]
\[
U(n) \cong \{L \cong \mathbb{C}^n \subset \mathbb{C}^{n}\oplus \mathbb{C}^{n}: h\oplus-h|_{L}=0\}
\]
\[
Sp(n) \cong \{L \cong \mathbb{H}^n \subset \mathbb{H}^{n}\oplus \mathbb{H}^{n}: h\oplus-h|_{L}=0\}
\]
where $h(u,v)=\bar{u}^tv$ on $\mathbb{R}^{n}$, $\mathbb{C}^{n}$, $\mathbb{H}^{n}$.
More generally, \cite{HL} provides a unified geometric realization of all compact Riemannian symmetric spaces as Grassmannians using the magic square.

In \cite{He}, H.Y. He established that every classical noncompact Lie group admits a natural compactification. Specifically, for $G=Aut(V,b)$, let $M=IG(V \oplus V, b\oplus-b)$ denote the Grassmannian of maximal isotropic $\mathbb{F}$-subspaces in $V\oplus V$ with respect to the bilinear form $b\oplus-b$.  Then $\overline{G} \triangleq M$ defines the Grassmannian compactification of $G$. In particular, $\overline{GL(n, \mathbb{F})}=Gr_{\mathbb{F}}(n,2n)$. B.H. Fu and Q.F. Li show in \cite{FL} that the wonderful compactification can be obtained geometrically via successive blowups of the Grassmannian compactification.

Our first result establishes a Grassmannian-type realization of every group $G$ and its Lie algebra $\mathfrak{g}$.  Specifically, we show that
$G$ can be identified with a double graph-like Grassmannian. When $G$ is compact, every $L \in M$ is double graph-like and the equality $\overline{G}=G$ holds automatically (Proposition 5.1).
Moreover, $\mathfrak{g}$ corresponds to a graph-like Grassmannian.

\begin{thm}
(i) $G \cong \{L \in M: L \bigcap (V\oplus 0)=L \bigcap (0\oplus V)=\{0\}\}$;

(ii) $\mathfrak{g} \cong \{L \in M: L \bigcap L_0=\{0\}\}$ for any $L_0 \in M$.
\end{thm}
This Grassmannian viewpoint on $G$  allows us to interpret the group structure of $G$ in terms of the geometry of configurations of subspaces in $V\oplus V$ (see Appendix),
it also provides a natural description of the Cartan involution on  $G$ within the same geometric framework.

\bigskip

The Cartan involution is a fundamental structure in Lie theory, representation theory, harmonic analysis, and Riemannian symmetric space theory.
For any classical noncompact Lie group $G=Aut(V,b)$, given an inner product on $V$, it determines

(1) a maximal compact subgroup $K \subset G$ preserving this inner product, and a Cartan involution  $\rho: G \rightarrow G$ with fixed point set $G^{\rho}=K$ and;

(2) a Riemannian structure $d_{\overline{G}}$ on $\overline{G}$, realizing $\overline{G}$ as a Riemannian symmetric space.
\begin{thm}
The Cartan involution $\rho$ on $G$ extends uniquely to an involutive isometry $\bar{\rho}$ on $(\overline{G}, d_{\overline{G}})$ and $\overline{G}^{\bar{\rho}}=G^{\rho}=K$.
\end{thm}
Noncompact Riemannian symmetric spaces are in one-to-one correspondence with  noncompact reductive Lie groups \footnote{up to  finite covers.}.
For a noncompact  $G=Aut(V,b)$,  consider the involution $\eta: G \rightarrow G$ defined by  $\eta(g)=\rho(g)^{-1}$.
Then the identity connected component $(G^{\eta})_0$ of the fixed point set $G^{\eta}$ is isomorphic to the associated noncompact Riemannian symmetric space
$G/K$, i.e.
 \[
 (G^{\eta})_0 \cong G/K.
 \]
 Moreover, the following theorem holds:
\begin{thm}
(i) $\eta$ extends uniquely to an involutive isometry $\bar{\eta}$ on $(\overline{G}, d_{\overline{G}})$.

(ii) $\overline{G}^{\bar{\eta}}$ and $\overline{G}^{\bar{\rho}}$ form a complementary  pair of reflective submanifolds in $\overline{G}$. Consequently, there exists a compact Lie group $G_c$ such that $\overline{G}^{\bar{\eta}} \cong G_c/K$.
\end{thm}
Recall that a submanifold $N$ of a Riemannian manifold $M$ is said to be reflective if it is the fixed point set
of an involutive isometry on $M$. Two reflective submanifolds $N_1$ and $N_2$ passing through the origin point $o$ of a Riemannian symmetric space $M$
form a complementary  pair if and only if $T_oN_2$ is the orthogonal complement of $T_oN_1$ in $T_oM$, i.e. $T_oM=T_oN_1\oplus T_oN_2$. The classification of reflective submanifolds
and complementary pairs of reflective submanifolds in Riemannian symmetric spaces was established by S.P. Leung in \cite{L2}.

As a corollary, we obtain the embedding $G/K \cong (G^{\eta})_0 \hookrightarrow \overline{G}^{\bar{\eta}} \cong G_c/K$, which generalizes the Borel embedding (known in the Hermitian case) to all classical Riemannian symmetric spaces. Furthermore, this construction induces a natural embedding of certain semisimple symmetric spaces into Grassmannian manifolds.
Additionally, we characterize classical noncompact Riemannian symmetric spaces as ``space-like" Grassmannians; that is, they consist exactly of the space-like elements within their compact duals.

\begin{thm}
$G/K  \cong \{L \in G_c/K: h\oplus-h|_L>0\}$.
\end{thm}

\begin{rmk}
In this paper, we restrict our attention to classical Lie groups. A natural question arises: can the construction be extended to exceptional Lie groups? In the classical setting,
$G$ acts on $\mathbb{F}^n$ and $\overline{G}$ consists of
certain subspaces in $\mathbb{F}^{2n}$, where $\mathbb{F}=\mathbb{R},\mathbb{C},\mathbb{H}$. If one attempts to apply the same construction to exceptional Lie groups,
the group $G$ would act on $\mathbb{O}^n$ and $\overline{G}$ would be formed by subspaces in $\mathbb{O}^{2n}$. However, geometries related to $\mathbb{O}^n$ is only well-defined for $n\leq 3$, which forces $n=1$
in the case of $\mathbb{O}^n$
and $\mathbb{O}^{2n}$. Consequently, our previous framework can only be applied to the exceptional group
$G_2$. Indeed, related results concerning the compactification of $G_2$ are available in the literature \cite{KP}\cite{Man2}, by S.Y. Kim,  K.D. Park and L. Manivel.
\end{rmk}

The paper is structured as follows: Section 2 provides Grassmannian descriptions of $G$, $\mathfrak{g}$ and $\overline{G}$;
Section 3 extends the Cartan involution on $G$ to $\overline{G} $ and interprets it within the framework of Grassmannian geometry;
Section 4 investigates noncompact Riemannian symmetric spaces from a Grassmannian viewpoint.
In the Appendix, the group structure of $G$ is interpreted through the geometry of subspaces.

\bigskip

$\mathbf{Acknowledgements.}$  The first author is supported by the National Natural
Science Foundation of China (No. 12471045). The second author is
supported by  research grants from the Research Grants Council of the Hong
Kong Special Administrative Region, China (No. 14302224, No. 14305923, No. 14306322).

\bigskip

\section{{\large Grassmannian perspectives of classical Lie groups}}
\subsection{General linear groups}
We begin by considering the general linear groups $G=GL(n,\mathbb{F})=Aut(V)$, where $\mathbb{F}=\mathbb{R},\mathbb{C},\mathbb{H}$ and $V \cong \mathbb{F}^n$ denotes the corresponding
$n$-dimensional vector space over $\mathbb{F}$ (or $\mathbb{F}$-module when $\mathbb{F}=\mathbb{H}$). Let $V_1$ and $V_2$  be two copies of $V$, and define $M=Gr(n, V_1\oplus V_2)$ as the
Grassmannian of $n$-dimensional $\mathbb{F}$-linear subspaces in $V_1 \oplus V_2$. Recall that
$M$ carries the structure of a Riemannian symmetric space. In the three cases over $\mathbb{R}$, $\mathbb{C}$ and $\mathbb{H}$, they are respectively isomorphic to $O(2n)/O(n) \times O(n)$, $U(2n)/U(n) \times U(n)$, and $Sp(2n)/Sp(n) \times Sp(n)$.

By the standard geometry of Grassmannians, for any point $X \in M$, the tangent space at $X$  is intrinsically described as $T_{X}M=Hom(X, (V_1 \oplus V_2)/X)$.
To obtain a more explicit representation, we endow $V$ (and consequently $V_1 \oplus V_2$) with a standard inner product. This choice induces a canonical orthogonal decomposition
$V_1 \oplus V_2=X \oplus X^{\perp}$, which in turn yields an isomorphism $(V_1 \oplus V_2)/X \cong X^\perp$, where
$X^{\perp}$ denotes the orthogonal complement of $X$ in $V_1\oplus V_2$.

\bigskip
Notations: Throughout this paper, $X^{\perp}$ always denotes the orthogonal complement of $X$ in $V_1\oplus V_2$ with respect to the standard inner product on $V_1\oplus V_2$.

\bigskip
Hence, the tangent space can be equivalently written as $T_{X}M=Hom(X, X^{\perp})$. Consider the natural embedding
\[
f: Hom(X,X^\perp) \hookrightarrow M, ~  \varphi \mapsto
graph(\varphi):=\{x + \varphi(x): x \in X\}
\]
we obtain the following identification \cite{C}:
\[
T_{X}M=Hom(X, X^{\perp}) \cong \{L \in
M: L \bigcap X^\perp=\{0\}\}.
\]

We define the associated ``graph-like Grassmannian" as
\[
M_{X} \triangleq \{L \in
M: L \bigcap X^{\perp}=\{0\}\}.
\]
The terminology reflects that every $ L \in M_X $ corresponds uniquely to the graph of a linear map from $ X $ to $ X^{\perp} $.
It is obvious that:
\begin{lem}
$M_{X} \cong \mathfrak{g}$, where $\mathfrak{g}$ is the Lie algebra of $G$.
\end{lem}

For any pair $X, Y \in M$ satisfying  $X \cap Y=\{0\}$,  we define the ``double graph-like Grassmannian" as
\[
M_{XY} \triangleq \{L \in M: L \bigcap X^\perp=L \bigcap Y^\perp=\{0\}\}.
\]
The terminology is motivated by the fact that any $L \in M_{XY}$ can be simultaneously represented as the graph of a linear map from
$X$ to $X^\perp$ and as the graph of a linear map from $Y$ to $Y^\bot$. Consequently, we obtain the identification

\begin{thm}
$M_{XY} \cong G$.
\end{thm}
\begin{proof}
The proof, together with the description of the group structure on $M_{XY}$ in terms of the geometry of subspace configurations in $V_1\oplus V_2$, is provided in the Appendix.
\end{proof}

\begin{rmk}
In fact, the inner product is not required to recover the group $G$ inside $M$. Briefly, for any pair $X, Y \in M$ satisfying  $X \cap Y=\{0\}$,  one may define the double graph-like Grassmannian as
\[
\widetilde{M}_{XY}  \triangleq \{L \in M: L \bigcap X=L \bigcap Y=\{0\}\}.
\]
Then  each $ L \in \widetilde{M}_{XY} $ corresponds uniquely to a graph  from $ X $ to $Y $ and also to a graph from $Y$ to $X$.

For any $L \in M$, we may decompose it with respect to the subspaces  $X$ and $Y$. Moreover, if $X \in \widetilde{M}_{XY}$, the decomposition yields an isomorphism
\[
\varphi_{L} :  ~  ~  X  ~  ~  \xrightarrow[]{\sim}  ~  ~  L   ~  ~   \xrightarrow[]{\sim}   ~   ~ Y.
\]
Using arguments analogous to those in the Appendix, we obtain $\widetilde{M}_{XY} \cong G$.
\end{rmk}

In particular, we have the identification $G \cong M_{V_1V_2}$ and any $g \in G$ can be viewed as a graph from $V_1$ to $V_2$. This yields a natural embedding:
\[
\mathfrak{i}: G \hookrightarrow M, ~  \mathfrak{i}(g)= L_g:=\{(x, gx): x \in V\}.
\]
The image $i(G)$ is an open dense subset of the compact Grassmannian $M=Gr(n, V_1\oplus V_2)$.  Consequently,
$M$ defines a compactification of the noncompact group  $G=Aut(V)$. We denote this compactification by $\overline{G} \triangleq M$, and refer to it as the Grassmannian compactification of
$G$. Its properties have been extensively studied in \cite{FL}\cite{GGKW}\cite{He}.

\subsection{Other classical Lie groups}
We continue to consider other classical Lie groups $G=Aut(V,b)$, where $V \cong \mathbb{F}^n$ ($\mathbb{F}=\mathbb{R},\mathbb{C},\mathbb{H}$) with $b$ a non-degenerate $\mathbb{R}$-bilinear form on it.
Concretely, $G=Aut(V,b)$  belongs to one of the following seven types\cite{Har}\cite{K} ($n=p+q=2m$):

\hspace{-2mm}$V=\mathbb{R}^{n}, b(u,v)=u^tI_{p,q}v, O(p,q)=\{g \in GL(p+q,\mathbb{R})| g^tI_{p,q}g=I_{p,q}\}$

\hspace{-2mm}$V=\mathbb{C}^{n}, b(u,v)=\bar{u}^tI_{p,q}v, U(p,q)=\{g \in GL(p+q,\mathbb{C})| \bar{g}^tI_{p,q}g=I_{p,q}\}$

\hspace{-2mm}$V=\mathbb{H}^{n}, b(u,v)=\bar{u}^tI_{p,q}v, Sp(p,q)=\{g \in GL(p+q,\mathbb{H})| \bar{g}^tI_{p,q}g=I_{p,q}\}$

\hspace{-2mm}$V=\mathbb{R}^{n}, b(u,v)=u^tJ_{m,m}v, Sp(2m,\mathbb{R})=\{g \in GL(2m,\mathbb{R})| g^tJ_{m,m}g=J_{m,m}\}$

\hspace{-2mm}$V=\mathbb{C}^{n}, b(u,v)=u^tJ_{m,m}v, Sp(2m,\mathbb{C})=\{g \in GL(2m,\mathbb{C})| g^tJ_{m,m}g=J_{m,m}\}$

\hspace{-2mm}$V=\mathbb{C}^{n}, b(u,v)=u^tv, O(n,\mathbb{C})=\{g \in GL(n,\mathbb{C})| g^tg=I\}$

\hspace{-2mm}$V=\mathbb{H}^{n}, b(u,v)=\bar{u}^tiv, O^*(2n)=\{g \in U(m,m)| g^tI_{m,m}J_{m,m}g=I_{m,m}J_{m,m}\}$
here $I_{p,q}=\left(
\begin{array}{cc}
I_p &  \\
  & -I_q\end{array} \right)$, $J_{m,m}=\left(
\begin{array}{cc}
  &  I_m\\
-I_m  & \end{array} \right)$. Note that for any $g \in G$, the elements  $\bar{g}$, $g^t $ and $\bar{g}^t $ also belong to $G$.

\bigskip

Let $(V_1,b)$ and $(V_2,b)$  be two copies of $(V,b)$, and denote $M=IG(V_1 \oplus V_2, b\oplus-b)$ as the Grassmannian of maximal isotropic subspaces in $V_1 \oplus V_2$ with respect to the form $b\oplus-b$, i.e.
\[
M=IG(V_1 \oplus V_2, b\oplus-b)=\{L \in Gr(n, V_1\oplus V_2): b\oplus-b|_{L}=0\}.
\]

As shown in Theorem 0.3 of \cite{He}, we obtain the following result:
\begin{prop} $M$ is a Riemannian symmetric space and
\begin{enumerate}
    \item $M \cong O(p+q)$ for $G = O(p, q)$;
    \item $M \cong U(p + q)$ for $G = U(p,q)$;
    \item $M \cong Sp(p + q)$ for $G = Sp(p,q)$;
    \item $M \cong U(2m)/O(2m)$ for $G = Sp(2m, \mathbb{R})$;
    \item $M \cong Sp(2m)/U(2m)$ for $G = Sp(2m, \mathbb{C})$;
    \item $M \cong O(2n)/U(n)$ for $G = O(n,\mathbb{C})$;
    \item $M \cong U(2m)/Sp(m)$ for $G = O^*(2m)$.
\end{enumerate}
\end{prop}
\begin{proof}
We give an alternative proof of this lemma here. In \cite{HL}, the second author gave a uniform description of compact Riemannian symmetric spaces as generalized Grassmannians. From the classification list of \cite{HL}, all
the above $M=IG(V_1 \oplus V_2, b\oplus-b)$ appear as generalized Grassmannians and their realizations as quotients of Lie groups are listed there. As an illustration, consider case
(4): $M \cong U(2m)/O(2m)$ for $G = Sp(2m, \mathbb{R})$. In this case
\[
M=IG(V_1 \oplus V_2, b\oplus-b)=\{L \in Gr(2m, 4m): L^t diag(J_{m,m},-J_{m,m})L=0\}
\]
\[
U(2m)/O(2m) \cong \{L \in Gr(2m, 4m): L^t J_{2m,2m}L=0\}
\]
Since $diag(J_{m,m},-J_{m,m})$ is similar to $J_{2m,2m}$, we obtain $M \cong U(2m)/O(2m)$.
\end{proof}

Also in Lemma 2.1 of \cite{He}, the author showed that:
\begin{lem}
For any $L \in M$, $\dim(L \bigcap V_1)=\dim(L \bigcap V_2)$.
\end{lem}

Similar to the case of general linear groups, for any point $X \in M$, the tangent space at $X$ is given by  \cite{C}:
\[
T_{X}M \cong \{L \in
M: L \bigcap X^\perp=\{0\}\}.
\]
We define the associated ``graph-like Grassmannian"  as
\[
M_{X} \triangleq \{L \in
M: L \bigcap X^{\perp}=\{0\}\}.
\]
Then the following lemma holds:

\begin{lem}
$M_{X} \cong \mathfrak{g}$, where $\mathfrak{g}$ is the Lie algebra of $G$.
\end{lem}
\begin{proof}
Since $M$ is a Riemannian symmetric space, there is a natural isomorphism between tangent spaces of two different points.
WLOG, we can take $X=\{(x,x): x \in V\}$, then $X^\perp=\{(x,-x): x \in V\}$. For any $n$-dimensional subspace $L$ in $V_1 \oplus V_2$,
$ L \bigcap X^\perp=\{0\}$ is equivalent to $L$ can be written as\footnote{To avoid confusion with the notation $L_g:=\{(x, gx): x \in V\}$, we use $L'_A$ here.} $L'_A=\{(x+Ax,x-Ax): x \in V\}$ for some $A \in M_{n \times n}(\mathbb{F})$.
By direct computation,
\[
b\oplus-b|_{L'_A}=0 \Longleftrightarrow A \in \mathfrak{g}
\]
Hence we have
\[
M_{X}=\{L'_A=\{(x+Ax,x-Ax): x \in V\}:  A \in \mathfrak{g} \} \cong \mathfrak{g}.
\]
\end{proof}

\begin{rmk}
For any $X \in M$, we have $X^\perp \in M$. Therefore, we may also express the Lie algebra as $\mathfrak{g} \cong \{L \in M: L \bigcap L_0=\{0\}\}$ for any $L_0 \in M$.
\end{rmk}

\bigskip
For any pair $X, Y \in M$ satisfying  $X \cap Y=\{0\}$, the set $M_{XY} \triangleq \{L \in M: L \bigcap X^\perp=L \bigcap Y^\perp=\{0\}\}$
is, in contrast to the general linear group case, not isomorphic to $G$. For example, if we take  $X=\{(x,x): x \in V\}$, $Y=\{(x,-x): x \in V\}$, then
$M_{XY}=\{L'_A=\{(x+Ax,x-Ax): x \in V\}:  A \in \mathfrak{g}, ~ A ~ invertible \}$,
which is not isomorphic to $G$.

Instead, we introduce the double graph-like Grassmannian
\[
M_{V_1V_2} \triangleq \{L \in M: L \bigcap V_1=L \bigcap V_2=\{0\}\}.
\]
(Note that $V_1, V_2 \notin M$.) For this space we obtain the following isomorphism:

\begin{thm}
$M_{V_1V_2} \cong G$.
\end{thm}
\begin{proof}
See the Appendix for the proof.
\end{proof}

Similar to the case of general linear groups, from $G \cong M_{V_1V_2}$, any $g \in G$ can be viewed as a graph from $V_1$ to $V_2$. This yields a natural embedding:
\[
\mathfrak{i}: G \hookrightarrow M, ~  \mathfrak{i}(g)= L_g:=\{(x, gx): x \in V\}.
\]
The image $i(G)$ is an open dense subset of the compact Grassmannian $M=IG(V_1 \oplus V_2, b\oplus-b)$.  Consequently,
$M$ defines a compactification of the noncompact group  $G=Aut(V,b)$. We denote this compactification by $\overline{G} \triangleq M$, and refer to it as the Grassmannian compactification of $G$. Its properties have been extensively studied in \cite{FL}\cite{GGKW}\cite{He}.

\bigskip
\section{{\large Grassmannian perspectives of Cartan involutions}}
Let $G$ be a classical noncompact Lie group and $\overline{G}=M $ be its Grassmannian compactification.
In this section, we provide a description of the Cartan involution on $G$ in terms of Grassmannian geometry by extending it to $\overline{G} $.
\subsection{General linear groups}
Let $G=Aut(V)=GL(n,\mathbb{F})$ and $M=Gr(n, V_1\oplus V_2)$ as previously defined. On the vector space $V \cong \mathbb{F}^n$, we fix the standard inner product $h(u,v)=\bar{u}^tv$ for any $u,v \in \mathbb{F}^n$. Then
$K \triangleq Aut(V,h)= \{g \in G: h(gu,gv)=h(u,v), ~   \forall u,v \in V \}$ is a maximal
compact subgroup of $G$. Specifically,  $K=O(n),U(n),Sp(n)$ for $\mathbb{F}=\mathbb{R},\mathbb{C},\mathbb{H}$ respectively. The form $h$ will also induce an involutive automorphism of $G$, i.e. $\rho: G \rightarrow G$, $g \mapsto \rho(g)$ such that $h(u,v)=h(gu,\rho(g)v)$ for any $u,v \in V$.
It is obvious that $\rho(g)=(\bar{g}^t)^{-1}$ and it  is the standard Cartan involution on $G$ with fixed point set $G^{\rho}=K$.

On $V_1\oplus V_2$, we equip it with the direct sum inner product $h\oplus h$, which induces a Riemannian structure on the Grassmannian $M=Gr(n, V_1\oplus V_2)$. Explicitly, for any $X,Y \in M$, the geodesic distance between them is given by
$d(X,Y)=(\sum_{i=1}^{i=n} \theta_{i}^{2})^{\frac{1}{2}}$, here $0\leq \theta_{1} \leq \cdots \leq \theta_{n} \leq\frac{\pi}{2}$ are the principal angles between $X$ and $Y$
with respect to $h\oplus h$ \cite{Wong}\cite{YL}.

On $V_1\oplus V_2$, we also consider the indefinite inner product $\mathbf{h} =h\oplus -h$. Since $\mathbf{h}$ is non-degenerate, we have $(L^{\perp_{\mathbf{h}}})^{\perp_{\mathbf{h}}}=L$ for any $L \in M$, where $L^{\perp_{\mathbf{h}}}$
denotes the orthogonal complement of  $L$ in $V_1\oplus V_2$ with respect to $\mathbf{h}$.
Therefore, $\mathbf{h}$ induces an involution on $M$:
\[
\bar{\rho}: M \rightarrow M, ~  \bar{\rho}(L)=L^{\perp_{\mathbf{h}}}.
\]

\begin{thm}
(i) $\bar{\rho}$ is an involutive isometry on $M$ with respect to $d$.

(ii) The fixed point set of $M$ under $\bar{\rho}$, i.e. $M^{\bar{\rho}}=\{L \in M: \mathbf{h}|_L=0\}$, is a reflective submanifold of $M$, hence it is again a Riemannian symmetric space.

(iii) The isometry $\bar{\rho}$ is the unique extension to $M$ such that  $\bar{\rho}|_G=\rho$.

(iv) $M^{\bar{\rho}}=G^{\rho}=K$.
\end{thm}
\begin{proof}
(i) For any $L \in M$, since
$\mathbf{h}(L, I_{n,n}L^{\perp})=\bar{L}^tI_{n,n}I_{n,n}L^{\perp}=\bar{L}^tL^{\perp}=0$, we conclude that $\bar{\rho}(L)=I_{n,n}L^{\perp}$. For
any $L_1, L_2 \in M$,
\[
d(\bar{\rho}(L_1), \bar{\rho}(L_2))=d(I_{n,n}L_1^{\perp},I_{n,n}L_2^{\perp})=d(L_1^{\perp},L_2^{\perp})=d(L_1,L_2),
\]
which shows that  $\bar{\rho}$ is an involutive isometry on $M$.

(ii) This follows directly from (i). And according to \cite{L2}, reflective submanifolds of Riemannian symmetric spaces are again Riemannian symmetric spaces.

(iii) Recall the embedding $i: G\hookrightarrow M$  defined by $g \mapsto L_g$. From (i), we obtain $\bar{\rho}(L_g)=L_{(\bar{g}^t)^{-1}}=L_{\rho(g)}$, which shows that $\bar{\rho}$  restricts to $G$ and satisfies
$\bar{\rho}|_G=\rho$.
The uniqueness of $\bar{\rho}$ follows  since  $G$ is dense in $M$.

(iv) Since $G=\{L \in M: L \bigcap V_1=L \bigcap V_2=\{0\}\} \subset M$ and $G^{\rho}=K$, it suffices to show that for any $L \in M^{\bar{\rho}}$, we have $L \bigcap V_1=L \bigcap V_2=\{0\}$.
This follows from the fact that $\mathbf{h}(L,L)=0$ for any $L \in M^{\bar{\rho}}$, while $\mathbf{h}(V_1,V_1)>0$ and $\mathbf{h}(V_2,V_2)<0$.
\end{proof}

\begin{rmk}
Note that the Cartan involution on $G$ is not unique. Analogous results hold for other choices of Cartan involutions. In general, a Cartan involution can be written as
$\rho_p(g)=P^{-1}(\bar{g}^t)^{-1}P$, where $P$ is a positive definite matrix. This involution corresponds to the inner product  $h_p(u,v)=\bar{u}^tPv$ for any $u,v \in V$. And
$Aut(V,h_p)=G^{\rho_p}=\{g \in G: \bar{g}^tPg=P\}$ is a maximal compact subgroup of $G$.

Similarly,  the inner product $h_p \oplus h_p$ induces a geodesic distance $d_{p}$ on $M$, while $\mathbf{h_p}=h_p \oplus -h_p$ defines an involution on $M$ via $\bar{\rho}_p(L)=L^{\perp_{\mathbf{h_p} }}$.
One can then verify that $\bar{\rho}_p$ is an involutive isometry on $M$ with respect to $d_p$, $\bar{\rho}_p$ is an extension of the Cartan involution $\rho_p$ and $M^{\bar{\rho}_p}=G^{\rho_p}$.
For simplicity, the subsequent discussion will focus only on the standard Cartan involution.
\end{rmk}

\subsection{Other classical Lie groups}
Let's continue to consider $G=Aut(V,b)$ and $M=IG(V_1 \oplus V_2, b\oplus-b)$ case. Our aim is to describe the Cartan involution on
$G$ in the language of Grassmannians. Recall
that the standard Cartan involution on $G$ is given by $\rho: G \rightarrow G$, $g \mapsto (\bar{g}^t)^{-1}$.
A natural question arises: can the involution
$\bar{\rho}$ defined earlier on $Gr(n, V_1 \oplus V_2)$ be restricted to $M=IG(V_1 \oplus V_2, b\oplus-b)$?
The answer is affirmative.

On $V_1 \oplus V_2$, we define the bilinear form $\mathbf{b}=b\oplus-b$, which induces an involution on $Gr(n, V_1\oplus V_2)$:
\[
\bar{\sigma}: Gr(n, V_1\oplus V_2) \rightarrow Gr(n, V_1\oplus V_2), ~  \bar{\sigma}(L)=L^{\bot_{\mathbf{b}}}.
\]

\begin{lem}
$\bar{\sigma}$ is an involutive isometry on $Gr(n, V_1\oplus V_2)$.
Hence the fixed point set $(Gr(n, V_1\oplus V_2))^{\bar{\sigma}}=IG(V_1 \oplus V_2, b\oplus-b)$  is a reflective submanifold of  $Gr(n, V_1\oplus V_2)$.
\end{lem}
\begin{proof}
From the list of $G=Aut(V,b)$ in section 2.2,  $\mathbf{b}(u,v)=u^t B v$ or $\mathbf{b}(u,v)=\bar{u}^t B v$ for any $u,v \in V_1\oplus V_2$,  where $B$ is the matrix representation of $\mathbf{b}$ and $B^{-1}=\bar{B}^t=\pm B$.

If $\mathbf{b}(u,v)=u^t B v$,
$\bar{\sigma}(L)=B^{-1}\bar{L}^{\perp}$ since $\mathbf{b}(L,B^{-1}\bar{L}^{\perp})=L^t B B^{-1} \bar{L}^{\perp}=0$.

If $\mathbf{b}(u,v)=\bar{u}^t B v$,
$\bar{\sigma}(L)=B^{-1}L^{\perp}$ since $\mathbf{b}(L,B^{-1}L^{\perp})=\bar{L}^tBB^{-1}L^{\perp}=0$.

In both cases, for any $L_1, L_2 \in Gr(n, V_1\oplus V_2)$, we have $d(\bar{\sigma}(L_1), \bar{\sigma}(L_2))=d(L_1, L_2)$, where $d$ is the geodesic distance on $Gr(n, V_1\oplus V_2)$. Hence $\bar{\sigma}$ is an isometry.
\end{proof}

By direct computations, we also obtain:
\begin{lem}
The involutions $\bar{\rho}$ and $\bar{\sigma}$ commute: $\bar{\rho}\bar{\sigma}=\bar{\sigma}\bar{\rho}$.
\end{lem}

As a corollary, we establish the following theorem:
\begin{thm}
(i) $\bar{\rho}$ restricts to an involution on $M=IG(V_1 \oplus V_2, b\oplus-b)$.

(ii) $\bar{\rho}$ is an involutive isometry on $M=IG(V_1 \oplus V_2, b\oplus-b)$.

(iii) The isometry $\bar{\rho}$ is the unique extension to $M$ such that  $\bar{\rho}|_G=\rho$.

(iv)  $M^{\bar{\rho}}=G^{\rho}=K$ is a maximal compact subgroup of $G$.
\end{thm}
\begin{proof}
(i) For any $L \in M$, $\bar{\sigma}(L)=L$, then $\bar{\sigma}(\bar{\rho}(L))=\bar{\rho}(\bar{\sigma}(L))=\bar{\rho}(L)$,
i.e. $\bar{\rho}(L) \in M$. Hence $\bar{\rho}$  restricts to $M=IG(V_1 \oplus V_2, b\oplus-b)$.

(ii) By \cite{QT}, reflective submanifolds of symmetric R-spaces are (geodesically) convex. Since $Gr(n, V_1\oplus V_2)$ is a symmetric R-space and $M=IG(V_1 \oplus V_2, b\oplus-b)$ is a reflective submanifold of it, we
have $d_{IG(V_1 \oplus V_2, b\oplus-b)}=d_{Gr(n, V_1\oplus V_2)}$. Therefore,  $\bar{\rho}$ remains an isometry on $M=IG(V_1 \oplus V_2, b\oplus-b)$.

(iii)--(iv) These follow directly from the definitions and earlier results.
\end{proof}

In conclusion, for any classical noncompact Lie group $G$, we present a Grassmannian description of the standard Cartan involution on $G$, i.e. for the double graph-like Grassmannian
$G= \{L \in M: L \bigcap V_1=L \bigcap V_2=\{0\}\}$, its standard Cartan involution is given by $\rho(L)=L^{\perp_{\mathbf{h}}}$ for any $L \in G$.

\bigskip
\section{{\large Grassmannian perspectives of classical noncompact Riemannian symmetric spaces}}
According to Cartan's classification of Riemannian symmetric spaces, there is a one-to-one correspondence between noncompact Riemannian symmetric spaces and noncompact reductive Lie groups. Specifically, for each noncompact Lie group $G$, one can associate a unique noncompact Riemannian symmetric space $G/K$, where $K=G^{\rho}\subset G$ is the fixed point set of a Cartan involution $\rho$ and constitutes a maximal compact subgroup of $G$.
Indeed, for a noncompact Lie group $G$,  consider the involution $\eta: G \rightarrow G$ defined by  $\eta(g)=\rho(g)^{-1}$. Then the identity connected component $(G^{\eta})_0$ of the fixed point set $G^{\eta}$ is isomorphic to the associated noncompact Riemannian symmetric space
$G/K$ \cite{Loos}.

In this section,  let $G$ be a classical noncompact Lie group and $\overline{G}=M $ be its Grassmannian compactification.
We will first provide a description of $\eta$ on $G$ in terms of Grassmannian geometry by extending it to $\overline{G} $.
Subsequently, we characterize classical noncompact Riemannian symmetric spaces as space-like Grassmannians.

\subsection{Grassmannian perspectives of $\eta$}
In the previous section, we provided a Grassmannian description of the standard Cartan involution $\rho: G \rightarrow G$, $g \mapsto (\bar{g}^t)^{-1}$
by extending it to an involuiton $\bar{\rho}: M \rightarrow M, L\mapsto L^{\perp_{\mathbf{h}}}$,  here $\mathbf{h}(u,v)=\bar{u}^tI_{n,n}v$, $\forall u,v \in V_1\oplus V_2$ is a Hermitian form On $V_1\oplus V_2$.
Now we consider the involution  $\eta: G \rightarrow G$, $g \mapsto \bar{g}^t$. Note that $\eta$ is an involution, but it is not a group automorphism anymore. Can we give a Grassmannian description of this $\eta$?

On $V_1\oplus V_2$, we introduce a skew-Hermitian form $\mathbf{w}(u,v)=\bar{u}^tJ_{n,n}v$, $\forall u,v \in V_1\oplus V_2$.
Since $\mathbf{w}$ is non-degenerate, it induces an involution  on $Gr(n, V_1\oplus V_2)$:
\[
\bar{\eta}: Gr(n, V_1\oplus V_2) \rightarrow Gr(n, V_1\oplus V_2), ~  \bar{\eta}(L)= L^{\bot_\mathbf{w}}.
\]
Analogous to the case of  $\bar{\rho}$, we establish the following result:

\begin{thm}
(i) $\bar{\eta}$ restricts to an involution on $IG(V_1 \oplus V_2, b\oplus-b)$.

(ii) $\bar{\eta}$ is an isometry on $M=Gr(n, V_1\oplus V_2)$ or $IG(V_1 \oplus V_2, b\oplus-b)$.
The fixed point set of  $M$ under  $\bar{\eta}$, i.e. $M^{\bar{\eta}}=\{L \in M: \mathbf{w}|_L=0\}$, is a reflective submanifold of  $M$, hence it is again a Riemannian symmetric space.

(iii) The isometry $\bar{\eta}$ is the unique extension to $M$ such that  $\bar{\eta}|_G=\eta$.
\end{thm}

From this theorem, we obtain a Grassmannian interpretation of $\eta$. Specifically, on the double graph-like Grassmannian $G= \{L \in M: L \bigcap V_1=L \bigcap V_2=\{0\}\}$, the involution $\eta$ is given by $\eta(L)=L^{\bot_{\mathbf{w}}}$ for any $L \in G$.

\bigskip

Next we will compute the fixed point set of $\eta$ and  $\bar{\eta}$. Denote $H_n(\mathbb{F})$ to be the $n\times n$ Hermitian matrices and $H_n^{+}(\mathbb{F})$ to be the $n\times n$ positive-definite matrices.
Recall that $G^{\rho}=K$ is the maximal compact subgroup of $G$, by Loos \cite{Loos}:

\begin{prop}
$G^{\eta}=G \cap H_n(\mathbb{F})$ is a symmetric space in ``Loos" sense. The identity connected component of  $G^{\eta}$ is $(G^{\eta})_0=H_n^{+}(\mathbb{F}) \cap G=\{g\rho(g)^{-1}: g \in G\}=\{g\bar{g}^t: g \in G\} \cong G/K$.
\end{prop}
According to Loos \cite{Loos}, a manifold $M$ is called a symmetric space if it is endowed with the multiplication $\mu: M \times M \rightarrow M$, $(x,y)\mapsto x \cdot y$
satisfying the following axioms (1) $x \cdot x=x$, (2) $x \cdot (x \cdot y)=y$, (3) $x \cdot (y \cdot z)=(x \cdot y) \cdot (x \cdot z)$. In particular, both Riemannian symmetric spaces and Lie groups are examples of  “Loos” symmetric spaces. Furthermore, semisimple symmetric spaces, which generalize Riemannian symmetric spaces, also fall into this category.

Recall that
a semisimple symmetric space can be defined as a homogeneous space $G/H$,  where $G$ is a semisimple Lie group and $H$ is an open subgroup of the fixed point group of an involutive automorphism of $G$. Such spaces were classified by M. Berger based on Cartan's classification of Riemannian symmetric spaces \cite{B}.
Moreover, a connected semisimple symmetric space $G/H$ is Riemannian if and only if $H$ is compact \cite{R}.

\begin{rmk}
Other connected components of $G^{\eta}$ can be described in a similar manner.
For instance, when $G=GL(n,\mathbb{R})$, the connected components of $G^{\eta}$ are given by
\[
\{gI_{p,q}g^t: g \in GL(n,\mathbb{R})\}=GL(n,\mathbb{R})/O(p,q),
\]
each of which constitutes a semisimple symmetric space.
For more general cases, the corresponding connected components are also semisimple symmetric spaces,
which can be determined from the classification of semisimple symmetric spaces provided in \cite{B}\cite{Huang}.
\end{rmk}

\begin{prop}
$M^{\bar{\eta}}$ and $M^{\bar{\rho}}$ form a complementary  pair of reflective submanifolds in $M$. Consequently,  there exists a compact Lie group $G_c$ such that $M^{\bar{\eta}} \cong G_c/K$.
\end{prop}
According to \cite{L1}\cite{L2}, reflective submanifolds of a Riemannian symmetric space come in pairs. Explicitly, two reflective submanifolds $N_1$, $N_2$ through the origin point $o$ of $M$
is a complementary  pair if and only if $T_oN_1$ is the orthogonal complement of $T_oN_2$ in $T_oM$.
\begin{proof}
Both $M^{\bar{\rho}}$ and $M^{\bar{\eta}}$ are reflective submanifolds through the origin point $o=\{(x,x): x \in V\}$.
Let $\mathfrak{g}=\mathfrak{k}+\mathfrak{m}$ be the Cartan decomposition corresponding to the Cartan involution $\rho$. By arguments similar to
Lemma 2.3, we have
\[
T_{o}M \cong \{L'_A=\{(x+Ax,x-Ax): x \in V\}:  A \in \mathfrak{g} \} .
\]
\[
T_{o}M^{\bar{\rho}}  \cong \{L'_A=\{(x+Ax,x-Ax): x \in V\}:  A \in \mathfrak{k} \}.
\]
\[
T_{o}M^{\bar{\eta}} \cong \{L'_A=\{(x+Ax,x-Ax): x \in V\}:  A \in \mathfrak{m} \}.
\]
By identifying the tangent spaces at $o$ as graphs from  $o$ to $o^\perp$, we have $T_{o}M=\mathfrak{g}$, $T_{o}M^{\bar{\rho}}=\mathfrak{k}$, $T_{o}M^{\bar{\eta}}=\mathfrak{m}$.
Since $\mathfrak{k}$ is composed of skew-Hermitian elements in $\mathfrak{g}$ and $\mathfrak{m}$ is composed of Hermitian elements in $\mathfrak{g}$, $\mathfrak{k}$ is orthogonal to
$\mathfrak{m}$ with respect to the Frobenius inner product of matrices.
Hence $T_oM^{\bar{\eta}}$ is the orthogonal complement of $T_oM^{\bar{\rho}}$ in $T_oM$ and
 ($M^{\bar{\rho}}$, $M^{\bar{\eta}}$) is a complementary  pair.

Since $M^{\bar{\rho}}=K$, by the classification of complementary pairs of Riemannian symmetric spaces in \cite{L2}, we have $M^{\bar{\eta}} \cong G_c/K$
for a compact Lie group $G_c$.
\end{proof}

\begin{table}[htbp]
\centering
\caption{Fixed point set $G^{\rho}$, $M^{\bar{\rho}}$, $(G^{\eta})_0$ and $M^{\bar{\eta}}$}
\renewcommand\arraystretch{1.6} 
\begin{tabular}{|c|c|c|c|}
\hline
序号 & $G$ & $G^{\rho}= M^{\bar{\rho}}$ & $(G^{\eta})_0 \subset M^{\bar{\eta}}$ \\
\hline
1 & $\mathrm{GL}(n,\mathbb{R})$ & $\mathrm{O}(n)$ & $\mathrm{GL}(n,\mathbb{R})/\mathrm{O}(n)\subset \mathrm{U}(n)/\mathrm{O}(n)$ \\
\hline
2 & $\mathrm{GL}(n,\mathbb{C})$ & $\mathrm{U}(n)$ & $\mathrm{GL}(n,\mathbb{C})/\mathrm{U}(n)\subset \mathrm{U}(n)$ \\
\hline
3 & $\mathrm{GL}(n,\mathbb{H})$ & $\mathrm{Sp}(n)$ & $\mathrm{GL}(n,\mathbb{H})/\mathrm{Sp}(n)\subset \mathrm{U}(2n)/\mathrm{Sp}(n)$ \\
\hline
4 & $\mathrm{O}(p, q)$ & $\mathrm{O}(p)\mathrm{O}(q)$ & $\mathrm{O}(p, q)/\mathrm{O}(p)\mathrm{O}(q)\subset \mathrm{O}(p+q)/\mathrm{O}(p)\mathrm{O}(q)$ \\
\hline
5 & $\mathrm{U}(p, q)$ & $\mathrm{U}(p)\mathrm{U}(q)$ & $\mathrm{U}(p, q)/\mathrm{U}(p)\mathrm{U}(q)\subset \mathrm{U}(p+q)/\mathrm{U}(p)\mathrm{O}(q)$ \\
\hline
6 & $\mathrm{Sp}(p, q)$ & $\mathrm{Sp}(p)\mathrm{Sp}(q)$ & $\mathrm{Sp}(p, q)/\mathrm{Sp}(p)\mathrm{Sp}(q)\subset \mathrm{Sp}(p+q)/\mathrm{Sp}(p)\mathrm{Sp}(q)$ \\
\hline
7 & $\mathrm{Sp}(2m, \mathbb{R})$ & $\mathrm{U}(m)$ & $\mathrm{Sp}(2m, \mathbb{R})/\mathrm{U}(m)\subset \mathrm{Sp}(m)/\mathrm{U}(m)$ \\
\hline
8 & $\mathrm{Sp}(2m, \mathbb{C})$ & $\mathrm{Sp}(m)$ & $\mathrm{Sp}(2m, \mathbb{C})/\mathrm{Sp}(m)\subset \mathrm{Sp}(m)$ \\
\hline
9 & $\mathrm{O}(n,\mathbb{C})$ & $\mathrm{O}(n)$ & $\mathrm{O}(n,\mathbb{C})/\mathrm{O}(n)\subset \mathrm{O}(n)$ \\
\hline
10 & $\mathrm{O}^*(2m)$ & $\mathrm{U}(m)$ & $\mathrm{O}^*(2m)/\mathrm{U}(m)\subset \mathrm{O}(2m)/\mathrm{U}(m)$ \\
\hline
\end{tabular}
\end{table}

\subsection{Classical noncompact Riemannian symmetric spaces as space-like Grassmannians}
According to the general theory of Riemannian symmetric spaces, a duality exists between symmetric spaces of compact type and those of noncompact type.
As shown in Table 2 from the previous subsection, $(G^{\eta})_0 \cong G/K$ and $M^{\bar{\eta}} \cong G_c/K$ form a dual pair of Riemannian symmetric spaces.

Since classical noncompact Riemannian symmetric spaces are in one-to-one correspondence with classical noncompact Lie groups, the Table 2 above includes all ten types of classical noncompact Riemannian symmetric spaces.
Moreover, since $(G^{\eta})_0 \subset G^{\eta} \subset  M^{\bar{\eta}}$, we obtain the following proposition:

\begin{prop}
For any classical noncompact Riemannian symmetric space $G/K$ and its compact dual $G_c/K$, there exists a natural embedding of $G/K$ into $G_c/K$.
\end{prop}

Note that in the case of Hermitian symmetric spaces, the above mapping corresponds to the well-known Borel embedding theorem.

\begin{rmk}
By Remark 4.1, the other connected components of $G^{\eta}$  are, not necessary Riemannian, semisimple symmetric spaces. Therefore we also have a natural embedding from these
 semisimple symmetric spaces into the Grassmannian $M^{\bar{\eta}} \cong G_c/K$.
\end{rmk}

Next, we will describe the classical noncompact Riemannian symmetric spaces as ``space-like" Grassmannians. Recall that the embedding of $G$ intio $M$ is given by
\[
\mathfrak{i}: G \hookrightarrow M, ~  \mathfrak{i}(g)= L_g:=\{(x, gx): x \in V\}.
\]
Then by Proposition 4.1, we obtain
\[
G/K \cong (G^{\eta})_0 =\{L_g \in M: g \in G \cap H_n^{+}(\mathbb{F})\}.
\]
And by Proposition 4.2, we obtain
\[
G_c/K \cong M^{\bar{\eta}} =\{L \in M: \mathbf{w}|_L=0\}.
\]
Note that $g \in  H_n(\mathbb{F})$  implies $\mathbf{w}|_{L_g}=0$, hence $G/K \hookrightarrow G_c/K $.

\bigskip

For any $g \in H_n^{+}(\mathbb{F})$, the Cayley transform yields $C(g)=(I-g)(I+g)^{-1} \in H_n^{<1}(\mathbb{F})$, where $H_n^{<1}(\mathbb{F})$ denotes the set of
Hermitian matrices $A$ satisfying $I-\bar{A}^tA>0$. Motivated by this, we aim to define a ``Cayley transform" on the Grassmannnian $Gr(n,  V_1\oplus V_2)$.

To facilitate clarity, we represent a subspace $L \in Gr(n,  V_1\oplus V_2)$ via its matrix class. Explicitly, given a basis $\{\alpha_1, \alpha_2, \cdots, \alpha_n\}$ of $L$, we write
\[
L=[\alpha_1, \alpha_2, \cdots, \alpha_n]=\left[
\begin{array}{cc}
A \\
B \end{array} \right]
\]
with $A,B \in M_{n \times n}(\mathbb{F})$ and $rank(L)=n$. Here $\left[
\begin{array}{cc}
A_1 \\
B_1 \end{array} \right]=\left[
\begin{array}{cc}
A_2 \\
B_2 \end{array} \right]$ if and only if $\left(
\begin{array}{cc}
A_1 \\
B_1 \end{array} \right)=\left(
\begin{array}{cc}
A_2 \\
B_2 \end{array} \right)g$ for some $g \in GL(n, \mathbb{F})$.

We now define the ``Cayley transform" on  $Gr(n,  V_1\oplus V_2)$ as the map
\[
\mathcal{C}: Gr(n, V_1\oplus V_2) \rightarrow Gr(n, V_1\oplus V_2)
\]
\[
\mathcal{C}(\left[
\begin{array}{cc}
A \\
B \end{array} \right])=\left(
\begin{array}{cc}
I_n  &  I_n\\
I_n  & -I_n\end{array} \right)\left[
\begin{array}{cc}
A \\
B \end{array} \right]=\left[
\begin{array}{cc}
A+B \\
A-B \end{array} \right].
\]
Note that $ \mathcal{C}$ is invertible and $\mathcal{C}^2=Id$. This map is termed the Cayley transform because, for any
$g \in H_n^{+}(\mathbb{F})$, we have
\[
\mathcal{C}(L_g)=\left(
\begin{array}{cc}
I_n  &  I_n\\
I_n  & -I_n\end{array} \right)\left[
\begin{array}{cc}
I \\
g \end{array} \right]=\left[
\begin{array}{cc}
I+g \\
I-g \end{array} \right]=\left[
\begin{array}{cc}
I \\
(I-g)(I+g)^{-1} \end{array} \right]=L_{C(g)}
\]
where $C(g)=(I-g)(I+g)^{-1} \in H_n^{<1}(\mathbb{F})$  is the classical Cayley transform of the matrix $g$.

\begin{lem}
For any $L \in Gr(n,  V_1\oplus V_2)$, the following hold:

(i) $\mathbf{w}|_L=0 \Leftrightarrow \mathbf{w}|_{\mathcal{C}(L)}=0$.

(ii) $L=L_g  ~ with ~ g \in H_n^{+}(\mathbb{F})  \Leftrightarrow \mathbf{w}|_{\mathcal{C}(L)}=0, \mathbf{h}|_{\mathcal{C}(L)}>0$.
\end{lem}

\begin{proof}
(i) For any $L=\left[
\begin{array}{cc}
A \\
B \end{array} \right] \in Gr(n,  V_1\oplus V_2)$, we have
\[
\mathbf{w}|_L=0 \Leftrightarrow \bar{A}^tB=\bar{B}^tA \Leftrightarrow \mathbf{w}|_{\mathcal{C}(L)}=0
\]

(ii)For any $L=L_g=\left[
\begin{array}{cc}
I \\
g \end{array} \right] \in Gr(n,  V_1\oplus V_2)$, we have $\mathcal{C}(L_g)=L_{C(g)}$ and
\[
g \in H_n(\mathbb{F}) \Leftrightarrow \mathbf{w}|_{\mathcal{C}(L)}=0
\]
\[
g \in H_n^{+}(\mathbb{F})  \Leftrightarrow C(g) \in H_n^{<1}(\mathbb{F}) \Leftrightarrow \mathbf{w}|_{\mathcal{C}(L)}=0, \mathbf{h}|_{\mathcal{C}(L)}>0
\]
The proof is completed.
\end{proof}

Therefore, for the general linear group $G=Aut(V)$ and its compactification $M=Gr(n,  V_1\oplus V_2)$, we have the identifications:
\[
G_c/K \cong  \{L \in M: \mathbf{w}|_L=0\},
\]
\[
G/K \cong  \{L \in M: \mathbf{w}|_L=0, \mathbf{h}|_L>0\}.
\]

A natural question arises: what is the corresponding picture for  $G=Aut(V,b)$ and  $M=IG(V_1 \oplus V_2, b\oplus-b)$?

Note that in the $G=Aut(V,b)$ case, the Cayley transform $ \mathcal{C}$ does not preserve $M=IG(V_1 \oplus V_2, b\oplus-b)$, i.e. $\mathcal{C}(M)\neq M$.
Nevertheless, via the isomorphism
\[
\mathcal{C}: M \rightarrow \mathcal{C}(M)
\]
we obtain the realizations:
\[
G_c/K \cong  \{L \in \mathcal{C}(M): \mathbf{w}|_L=0\},
\]
\[
G/K \cong  \{L \in \mathcal{C}(M): \mathbf{w}|_L=0, \mathbf{h}|_L>0\}.
\]

In summary, we establish the following:

\begin{thm}
For any classical noncompact Riemannian symmetric space $G/K$ and its compact dual $G_c/K$, the space $G/K$ can be realized as the ``space-like" part of the Grassmannian $G_c/K$.
Explicitly, $G/K$ consists of points in $G_c/K$ which is space-like with respect to $\mathbf{h}$, i.e. $G/K  \cong \{L \in G_c/K: \mathbf{h}|_L>0\}$.
\end{thm}

\begin{rmk}
A similar result was obtained in \cite{C}\cite{CHL} using a different approach.
\end{rmk}

\bigskip
\section{{\large Appendix}}
In this appendix, we provide a description of the group structure on $G=Aut(V,b)$ via the geometry of the double graph-like Grassmannian.

To begin, we first conside $G=Aut(V)$, $M=Gr(n,V_1\oplus V_2)$ as before. Recall that for any $X, Y \in M$ with $X \cap Y=\{0\}$, we define ``double graph-like Grassmannian"  as
\[
M_{XY} \triangleq \{L \in M: L \bigcap X^\bot=L \bigcap Y^\bot=\{0\}\}.
\]
Denote by $\pi_X: V_1\oplus V_2 \rightarrow X$, $\pi_Y: V_1\oplus V_2 \rightarrow Y$  the canonical projections to $X$ and $Y$ respectively. Since
\[
\pi_{X}|_{L} ~ is ~ an ~ isomorphism \Leftrightarrow L \bigcap X^{\bot}=\{0\},
\]
\[
\pi_{Y}|_{L} ~ is ~ an ~ isomorphism \Leftrightarrow L \bigcap Y^{\bot}=\{0\},
\]
for any $L \in M_{XY}$, we have the following isomorphism:
\[
\varphi_{L} :  ~  ~  X  ~  ~  \xrightarrow[\sim]{(\pi_X|_{L})^{-1}}  ~  ~  L   ~  ~   \xrightarrow[\sim]{\pi_Y|_{L}}   ~   ~ Y
\]

\begin{lem}
For any $L_0, L_1, L_2 \in M_{XY}$, there exists a unique $L' \in M_{XY}$ such that
$\varphi_{L_2}\circ \varphi_{L_0}^{-1}\circ \varphi_{L_1}=\varphi_{L'}$.
\end{lem}
\begin{proof}
We will prove for $\mathbb{F}=\mathbb{R}$, the proof is similar for $\mathbb{F}=\mathbb{C},\mathbb{H}$. Note that this argument only depends on the
relative position of $X$ and $Y$. Since the action of $O(2n)$  on $M=Gr(n,V_1\oplus V_2)$
 is transitive and preserves the relative position of two subspaces, we can assume $X=\left[
\begin{array}{cc}
I \\
O \end{array} \right]$.

From $Y\bigcap X=\{0\}$, we have $Y=\left[
\begin{array}{cc}
A \\
I \end{array} \right]$ with $A$ an $n \times n$ matrix, otherwise we will have $Y=\left[
\begin{array}{cc}
* & *\\
* & 0_{n \times 1}\end{array} \right]$ after elementary column transformation, which is contradict to $Y\bigcap X=\{0\}$.

\bigskip

For any  $L \in M_{XY}$, $L \bigcap X^{\bot}=\{0\}$. Since $X^{\bot}=\left[
\begin{array}{cc}
O \\
I \end{array} \right]$, we have  $L=\left[
\begin{array}{cc}
I \\
B \end{array} \right]$ with $B$ an $n \times n$ matrix. Moreover,  since $Y^{\bot}=\left[
\begin{array}{cc}
I \\
-A^t \end{array} \right]$, we have $L \bigcap Y^{\bot}=\{0\}$ is equivalent to $A^t+B$ is invertible.

\bigskip

Under the above conditions, we have the map $\varphi_{L}$ is given by
\begin{align*}
\varphi_{L} : & ~  ~   ~ X  ~  ~   ~ \longrightarrow  ~  ~   ~ L  ~  ~   ~ ~  ~   ~\longrightarrow ~  ~   ~ ~  ~   ~ Y\\
 ~ ~ & \left(
\begin{array}{cc}
x \\
0 \end{array} \right) \mapsto  \left(
\begin{array}{cc}
x \\
Bx \end{array} \right) \mapsto  \left(
\begin{array}{cc}
A(I+A^tA)^{-1}(A^t+B)x \\
(I+A^tA)^{-1}(A^t+B)x \end{array} \right)
\end{align*}
where the first map from $X$ to $L$ is obvious and the second map from $L$ to $Y$ is given by the projection matrix $P_Y$ of $Y$. By direct computation, we have
\[
P_Y=Y(Y^tY)^{-1}Y^t=\left(
\begin{array}{cc}
A(I+A^tA)^{-1}A^t & A(I+A^tA)^{-1}\\
(I+A^tA)^{-1}A^t & (I+A^tA)^{-1}\end{array} \right).
\]

Now for any  $L_0=\left[
\begin{array}{cc}
I \\
B_0 \end{array} \right], L_1=\left[
\begin{array}{cc}
I \\
B_1 \end{array} \right], L_2=\left[
\begin{array}{cc}
I \\
B_2 \end{array} \right] \in M_{XY}$, we have $A^t+B_i$ is invertible for $i=0,1,2$ and $\varphi_{L_2}\circ \varphi_{L_0}^{-1}\circ \varphi_{L_1}$ is as follows:

\[
X  ~  ~   ~ \longrightarrow  ~  ~   ~ Y  ~  ~   ~ ~  ~   ~\longrightarrow ~  ~   ~ ~  ~   ~ X~  ~   ~ \longrightarrow  ~  ~   ~ Y
\]
\[
\left(
\begin{array}{cc}
x \\
0 \end{array} \right) \mapsto   \left(
\begin{array}{cc}
A(I+A^tA)^{-1}(A^t+B_1)x \\
(I+A^tA)^{-1}(A^t+B_1)x \end{array} \right)
\mapsto   \left(
\begin{array}{cc}
(A^t+B_0)^{-1}(A^t+B_1)x \\
0 \end{array} \right)
\]
\[
 \mapsto   \left(
\begin{array}{cc}
A(I+A^tA)^{-1}(A^t+B_2)(A^t+B_0)^{-1}(A^t+B_1)x \\
(I+A^tA)^{-1}(A^t+B_2)(A^t+B_0)^{-1}(A^t+B_1)x \end{array} \right)
\]
To get $\varphi_{L_2}\circ \varphi_{L_0}^{-1}\circ \varphi_{L_1}=\varphi_{L'}$, we just need to take $L'=\left[
\begin{array}{cc}
I \\
B' \end{array} \right]$ with $B'=(A^t+B_2)(A^t+B_0)^{-1}(A^t+B_1)-A^t$. Note that $A^t+B'=(A^t+B_2)(A^t+B_0)^{-1}(A^t+B_1)$ is invertible,
i.e. $L' \in M_{XY}$. And such $L'$ is unique.
\end{proof}

\begin{cor}
Fixing $L_0 \in M_{XY}$, define $f: M_{XY} \rightarrow Aut(X)$, $L \mapsto \varphi_{L_0}^{-1}\circ \varphi_{L}$, then $f$ is an isomorphism.
\end{cor}
\begin{proof}
It is obvious $f$ is injective. From the above lemma, we have group structure on $M_{XY}$, i.e. $M_{XY}$ is a subgroup of $Aut(X)$ under $f$.
Explicitly, for any $L_1, L_2 \in M_{XY}$, $f(L_1)\cdot f(L_2)=\varphi_{L_0}^{-1}\circ \varphi_{L_2}\circ \varphi_{L_0}^{-1}\circ \varphi_{L_1}=\varphi_{L_0}^{-1}\circ \varphi_{L'}=f(L')$ for some $L' \in M_{XY}$,
i.e. $f(M_{XY})$ is closed under multiplication; for any $L \in M_{XY}$, $f(L)^{-1}=(\varphi_{L_0}^{-1}\circ \varphi_{L})^{-1}=\varphi_{L_0}^{-1}\circ \varphi_{L_0}\circ \varphi_{L}^{-1}\circ \varphi_{L_0}=\varphi_{L_0}^{-1}\circ \varphi_{L'}=f(L')$ for some $L' \in M_{XY}$, i.e. $f(M_{XY})$ is closed under inverse.

Conversely, for any $g \in Aut(X)$, following the notations in the above lemma, we take $L=\left[
\begin{array}{cc}
I \\
B \end{array} \right]$ with $B=(A^t+B_0)g-A^t$ (recall that $X=\left[
\begin{array}{cc}
I \\
O \end{array} \right]$, $Y=\left[
\begin{array}{cc}
A \\
I \end{array} \right]$ and $L_0=\left[
\begin{array}{cc}
I \\
B_0\end{array} \right]$), then $f(L)=g$. Hence $f$ is surjective, $M_{XY}$ is isomorphic to $Aut(X)$.
\end{proof}

Therefore $G=Aut(V)=GL(n, \mathbb{F})$ can be realized as the
double graph-like Grassmannian
$M_{XY} \triangleq \{L \in M: L \bigcap X^\bot=L \bigcap Y^\bot=\{0\}\}$
for any $X, Y \in M$ with $X \cap Y=\{0\}$. A group structure is defined on $M_{XY}$ as follows:
for any $L_1, L_2 \in M_{XY}$, the product $L_1 \cdot L_2=L'$ is given by the condition $ \varphi_{L_2}\circ \varphi_{L_0}^{-1}\circ \varphi_{L_1}=\varphi_{L'}$.

\bigskip
In particular, taking $X=V_1$, $Y=V_2$ and denoting by $\pi_1: V_1\oplus V_2 \rightarrow V_1$, $\pi_2:V_1\oplus V_2 \rightarrow V_2$ the canonical projections.
Then for any $L=\left[
\begin{array}{cc}
A \\
B \end{array} \right] \in M$,
\[
\pi_1|_{L} ~ is ~ an ~ isomorphism \Leftrightarrow L \bigcap V_2=\{0\} \Leftrightarrow rank(A)=n,
\]
\[
\pi_2|_{L} ~ is ~ an ~ isomorphism \Leftrightarrow L \bigcap V_1=\{0\} \Leftrightarrow rank(B)=n.
\]
Hence, we obtain:
\begin{align}
M_{V_1V_2} & = \{L=\left[
\begin{array}{cc}
A \\
B \end{array} \right]: A ~ and ~ B ~ invertible ~ matrices\} \nonumber \\
  ~ & = \{L=\left[
\begin{array}{cc}
I \\
C \end{array} \right]: C ~ invertible ~ matrix\}   \nonumber \\
  ~ & \cong Aut(V) \nonumber
\end{align}
The group structure on $M_{V_1V_2}$ can be understood via the following construction. Explicitly, for any $L_1, L_2 \in M_{V_1V_2}$, the product $L'=L_1 \cdot L_2$ is determined by its corresponding isomorphism
$\varphi_{L'}$, which is depicted in the following diagram:

\[\includegraphics[scale=0.2]{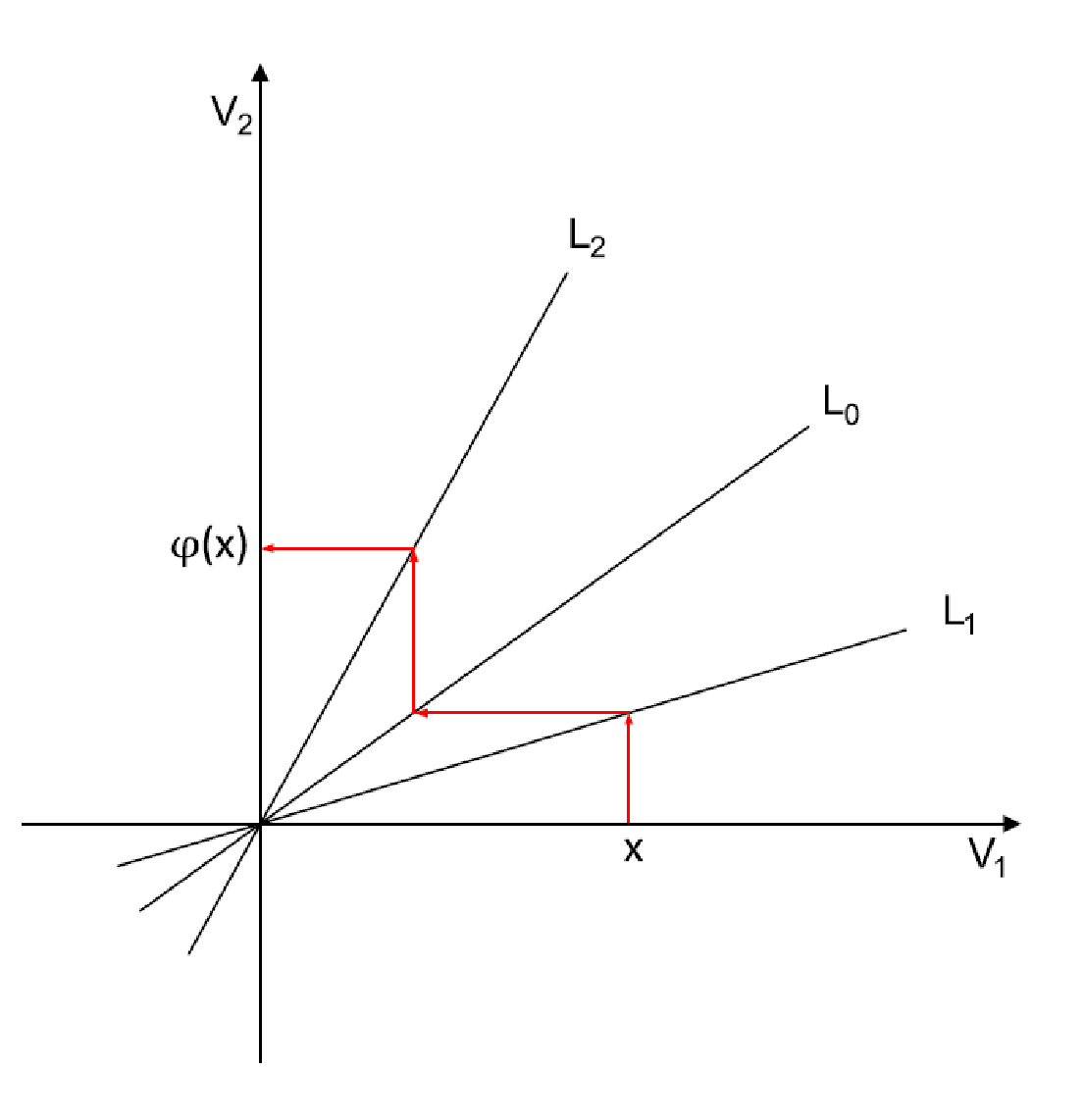}\]

\bigskip
We now turn to  other classical Lie groups $G=Aut(V,b)$. As before, define the isotropic Grassmannian  $M=IG(V_1 \oplus V_2, b\oplus-b)$ and consider the  double graph-like Grassmannian:
\[
M_{V_1V_2} \triangleq \{L \in M: L \bigcap V_1=L \bigcap V_2=\{0\}\}
\]
Then for any
$L \in M_{V_1V_2}$, we can write $L=\left[
\begin{array}{cc}
A \\
B \end{array} \right]$, where   $A$ and  $B$ are invertible  matrices  such that the restriction $b\oplus-b|_L=0$. It follows that
$L$ can be expressed in the form $L=\left[
\begin{array}{cc}
I \\
C \end{array} \right]$, where  $C \in Aut(V,b)$. Analogous to the  $G=Aut(V)$ case, the set $M_{V_1V_2}$ carries a group structure, and we have the isomorphism:
$M_{V_1V_2} \cong Aut(V,b)$.

\begin{prop}
Let $G=Aut(V,b)$ be compact. Then  for any $L \in M$,  we have $L \bigcap V_1=L \bigcap V_2=\{0\}$. Consequently, $\overline{G}=G$ holds automatically.
\end{prop}
\begin{proof}
We give the proof for $G=O(n)$, the arguments for $U(n)$ and $Sp(n)$ are analogous. Take $(V,b)$ with $V \cong \mathbb{R}^n$ and $b(u,v)=u^tv$. Then
\begin{align}
IG(\mathbb{R}^n \oplus \mathbb{R}^n, b\oplus-b) \triangleq  &  \{L \in Gr(n, \mathbb{R}^n \oplus \mathbb{R}^n) : b\oplus-b|_{L}=0\} \nonumber\\
=   & \{L=\left[
\begin{array}{cc}
A \\
B \end{array} \right]: A^tA=B^tB, rank(L)=n\} \nonumber \\
=   & \{L=\left[
\begin{array}{cc}
A \\
B \end{array} \right]: A^tA=B^tB, rank(A)=rank(B)=n\} \nonumber \\
=   & \{L=\left[
\begin{array}{cc}
I \\
C \end{array} \right]: C^tC=I\} \nonumber \\
\cong   & O(n) \nonumber
\end{align}
The key step, $rank(A)=rank(B)=n$, follows from the fact that in a real vector space, $u^tu=0$ implies $u=0$.

In contrast, for the case $(V,b)$ with $V \cong \mathbb{C}^n$ and $b(u,v)=u^tv$, we have
\begin{align}
IG(\mathbb{C}^n \oplus \mathbb{C}^n, b\oplus-b) \triangleq  &  \{L \in Gr(n, \mathbb{C}^n \oplus \mathbb{C}^n) : b\oplus-b|_{L}=0\} \nonumber\\
=   & \{L=\left[
\begin{array}{cc}
A \\
B \end{array} \right]: A^tA=B^tB, rank(L)=n\} \nonumber
\end{align}
but the implication $u^tu=0 \Rightarrow u=0$ fails over $\mathbb{C}$. Consequently, one cannot deduce
$rank(A)=rank(B)=n$ as in the real setting, which illustrates why the compact group equality $\overline{G}=G$ does not hold.
\end{proof}

\bigskip

School of Mathematics, East China University of Science and Technology, Shanghai 200237, China

E-mail address: yxchen76@ecust.edu.cn

\bigskip

The Institute of Mathematical Sciences and Department of Mathematics, The
Chinese University of Hong Kong, Shatin, N.T., Hong Kong

E-mail address: leung@ims.cuhk.edu.hk

\end{document}